\documentclass[12pt]{amsart}
 \pdfoutput=1
\usepackage{amsmath,amsthm,amscd,amssymb,latexsym,amsfonts,wasysym,xypic}
\usepackage{fancyhdr}                     
\usepackage{mathrsfs}                     
\usepackage{amsbsy}   
\usepackage{mathbbol}     

\usepackage{graphicx}     
\DeclareMathAlphabet{\mathpzc}{OT1}{pzc}{m}{it}

\usepackage{frcursive}

\usepackage[colorlinks]{hyperref}  

\usepackage[all]{hypcap}

\def\a{\alpha}
\def\b{\beta}

\def\d{\delta}

\def\e{\epsilon}
\def\f{\frac}                          
\def\g{\gamma}

\def\l|{\left|}
\def\la{\lambda}

\def\om{\omega}

\def\ov{\overline}                   

\def\r|{\right|}
\def\s{\sigma}

\def\un{\underline}                  

\def\z{\zeta}
\def\({\left(}
\def\){\right)}
\def\[{\left[}
\def\]{\right]}
\def\<{\left<}
\def\>{\right>}

\def\ra{\rightarrow}
\newcommand{\rz}{\raisebox{.2ex}{*}}

\def\SA{\mathcal A}
\def\SB{\mathcal B}

\def\SD{\mathcal D}

\def\SG{\mathcal G}

\def\SI{\mathcal I}

\def\SL{\mathcal L}

\def\SP{\mathcal P}

\def\SS{\mathcal S}

\def\Fb{\mathfrak b}
\def\Fc{\mathfrak c}

\def\Fj{\mathfrak j}

\def\Fl{\mathfrak l}

\def\Fo{\mathfrak o}

\def\Fr{\mathfrak r}
\def\Fs{\mathfrak s}
\def\Ft{\mathfrak t}
\def\Fu{\mathfrak u}
\def\Fv{\mathfrak v}
\def\Fw{\mathfrak w}

\def\fst{\frak{st}}

\def\pzp{\mathpzc{p}}

\def\scrb{\mathscr{B}}

\newcommand{\ccrs}[1]{\text{\scalebox{1}[.75]{\textcursive{#1}}}}

\def\bbn{{\mathbb N}}

\def\bbr{{\mathbb R}}

\def\bbz{{\mathbb Z}}

\def\ssm{\smallsetminus}

\newcommand{\bsm}[1]{\boldsymbol{#1}}    

\newtheorem{theorem}{Theorem}[section]
\newtheorem{lemma}{Lemma}[section]

\newtheorem{proposition}{Proposition}[section]

\newtheorem{definition}{Definition}[section]

\newtheorem{corollary}{Corollary}[section]

\theoremstyle{definition}
\newtheorem{remark}{Remark}[section]



\title{Nonstandard techniques and nowhere differentiable functions I: A dense family of generalized blancmange functions}

\author{Tom McGaffey}
\date{}

\begin{document}

\maketitle

\begin{abstract}
    We will give an elementary nonstandard proof that the family of generalized blancmange functions are nowhere differentiable. The proof follows from the intuitive characterization of differentiability at a point as almost $\delta$ affine along with the transfer of the functional equations these functions satisfy. We also give elementary nonstandard proofs of the uniform density of these functions among continuous functions. Finally, we discuss work done with the Python programming  language in displaying these functions.

\end{abstract}

\tableofcontents
   \pagenumbering{arabic} \setcounter{page}{1}   
\section{Introduction: Monsters and nonstandard characterization of differentiability}
   As far back as Bolzano, continuous nowhere differentiable functions  have been objects of fascination for mathematicians. Beginning sometimes during the first third of the 19th century, mathematicians began constructing  these functions (often called ``monsters'' in that earlier period) to understand, refine and contrast the notions of continuity and differentiability; all in a context where the very notion of function was in contention. For a perspective embedding the production of such ``pathological'' functions in the controversies over generality and rigor in the nineteenth century, see eg., the paper of Chorley, \cite{QuestionsGeneralityChorley2009}.  We became interested in these while reading the interesting study of mathematical conceptualization by Katz and Tall, \cite{KatzTensionIntuitiveInfinitesFormalMath}. Their infinitesimal microscopic perspective and discussion of the Takagi function, appropriately dubbed blancmange function, piqued the author's curiosity about  possible infinitesimal approaches to proving nowhere differentiability of functions defined in the manner of the blancmange function. We should note that with respect to properties of this specific nowhere differentiable function, there has been a wide range of investigations;  the paper of Allaart and Kawamura, \cite{TakagiSurvey2011}, is a good summary of this research.

   With some thought, the author realized that, using some elementary tools from nonstandard analysis, he could give an almost trivial  proof that the  blancmange function is nowhere differentiable. In particular, we will use no estimates of difference quotients. Instead, a use of the transfer of the  functional equations satisfied by this function  along with some elementary nonstandard tools  are sufficient to give this short proof. More specifically, we used the transfer of the sequence functional equations (see the first sentence in  \autoref{lem: *fcnal eqn and *affine on interval}) evaluated at an infinite index along with essentially crude order of magnitude algebraic characterizations of differentiability.

   The idea to analyze the functional equation at an infinite index is inspired by the author's recent awakening (due to the gentle prodding of Mikhail Katz) to the ingenious use of such ``tricks'' by Euler. (The recent paper \cite{TenMisconceptionsPublished} is a  good introduction to the important and accumulating historical works of M. Katz and his coauthors on eg., the early history of the calculus, including recent work on Euler in manuscript form.) We believe that the arguments in eg.,  \autoref{thm: 1st blancmange result} and \autoref{thm: 2nd blancmange result} were influenced by the exposure to Euler's remarkable facility with eg., infinite sums as long finite sums and orders of magnitude numerics in place of forbidden zones of ill defined products and quotients. Maybe the best place to see these displayed is his wonderful text \cite{EulerAnalysisInfinite}, where these brilliantly orchestrated strategies occur \emph{many} times. Note that Euler typically was no more than cryptically brief in his justifications of such gymnastics.  For our project, we think that viewing the infinite series defining the blancmange function as a `long finite sum' (and hence being able to apply the functional relation for infinitely long sums), as well as investigating the `end terms' beyond this long sum for simplifying manifestations was influenced by reading Euler. Note that the text of Kanovei and Reeken, \cite{Nonstandardanalysisaxiomatically}, gives an enlightening nonstandard rendition of Euler's proof of his famous product formula for sine (that appears eg., in the text of Euler already cited.) The ``nonstandard analysis'' text of Kanovei and Reeken and that of Gordon, Kusraev and Kutateladze, \cite{InfinitesAnalyGordonKusraevKutateladzeBk2002} contains several gems on the history of the calculus and eg., on Euler.

   We then realized that we could use  almost identical arguments to establish that  a wide variety of ``generalized blancmange functions'' are nowhere differentiable. In fact, we will show that our family, $\SB$, of continuous nowhere differentiable functions is dense in the space of continuous functions on $[0,1]$ with value $0$ at $0$ and $1$, see \autoref{cor: SB is dense}. Of course, it is an old standard fact, see eg., Thim's paper, that continuous nowhere differentiable functions are not only dense, but second countable. Our fact is much different (and is apparently new): it asserts the density of $\SB$, the set of functions defined via fractal type self-similarities on a set $\SS$,  of continuous piecewise linear functions. In other words, this is the family of such functions concretely defined in terms of a piecewise continuous function $s$ and a positive integer $c$ via  a sequence of self similar functional identities. (For the definitions of $\SS$ and $\SB$, see the constructions around
   \autoref{eqn: def of little s} and
   \autoref{eqn: def of SB}.) \emph{In summary, we believe that the import of this paper can be summarized as follows. First, we give a concrete construction of a dense family of continuous, nowhere differentiable  functions with large subfamilies having quite novel behaviors. Second, the proofs of nowhere differentiability (and density) are essentially order of magnitude algebraic arguments.   }

   Our primary references  on the technical history of such functions are the extensive master's thesis of Thim, \cite{ContNowhereDifferenFcns} which masterfully covers the technical history of these constructions, as well as the earlier paper of van Embe Boas, \cite{NowhereDifferenBoasMR0274670} giving some alternative perspectives on these constructions. In perusing the history of such functions in the papers of van Embe Boas and Thim, it appears that some of the nowhere differentiable functions constructed here have not been discussed before.

   We have on the one hand the wide variety of structural features of our generalized blancmange functions and on the other apparently only a handful of visual descriptions of continuous nowhere differentiable functions in the literature.
   So with the hope of supplementing this deficit, in the last section we will discuss work we did utilizing the Python programming language. Specifically, we  wrote code to display a sequences of magnifications of a tuple of approximations of an arbitrary generalized blancmange function. We will summarize the specifications of the codes as well as display two example (using much simpler code) with the intention of giving some impression of the diversity of these functions.


\section{Nonstandard preliminaries}
\subsection{Almost affine internal functions}
  We will assume the rudiments of (Robinsonian) nonstandard analysis; eg., elementary use of transfer for functions on Euclidean spaces and an isolated use of overflow not directly related to our proof. Good elementary introductions abound, eg., see the classic introduction of Lindstr{\o}m, \cite{Lindstrom1988}.
  The central idea underlying this section is the following.  To require that a functions $f:\bbr^m\ra\bbr^n$ be differentiable at $x\in\bbr^m$  is to require that, for each positive infinitesimal $\d$,  its restriction to $\bbr^m_\d(x)$ (the $\d$-module at $x$, see below)  visually looks like an affine map, at least up to magnitudes infinitely smaller than $\d$. This is the import of \autoref{prop: NS criterion for differen of f}. So to test for differentiability of a map at a point is to check that the map has such an almost affine structure for arbitrary positive infinitesimal $\d$. Below we will develop a few basic tools around this notion of almost affineness in order to exploit our criterion for differentiability in the following sections.

  We need some basic notation. Let $\rz\bbr$ denote the field of nonstandard real numbers and ${}^\s\bbr$ denote the external subfield isomorphic to the real numbers. Let $\rz\bbr_{nes}$ denote the subring of those that are nearstandard, ie., those $\Fr\in\rz\bbr$ that are infinitesimally close to a real number $a\in{}^\s\bbr$, denoted $\Fv\sim a$. Therefore, these are those nonstandard numbers $\Fv$ with a standard part, denoted $\fst(\Fv)$, in $\bbr$. It's basic that $\fst:\rz\bbr_{nes}\ra\bbr$ is a surjective ring homomorphism with kernel the ideal (in $\rz\bbr_{nes}$) of infinitesimals, $\mu(0)$, ie., those numbers $\d\sim 0$.
 \begin{definition}
    If $\Fr$ is a positive infinitesimal, we write $\bbr_\Fr$ for the $\rz\bbr_{nes}$ submodule of $\rz\bbr$ of all numbers $\Fv$ with $|\Fv|<a\Fr$ for some $a\in\bbr_+$, ie., $\Fr\rz\bbr_{nes}$. Of course, then $\bbr^m_\Fr$ will be the $\rz\bbr_{nes}$ submodule of $\rz\bbr^m$ given by the $m$-fold Cartesian product of $\bbr_\Fr$. The $\rz\bbr_{nes}$-submodules $\bbr^k_\Fr<\rz\bbr^m$, for integers $k\leq m$ will be called $\bsm{\Fr}$\textbf{-subspaces of }$\bsm{\rz\bbr^m}$.
    If ${}^\s\bbr$ is the external subfield of standard numbers in $\rz\bbr$, then ${}^\s\bbr_\Fr$ will denote the external subring of $\bbr_\Fr$ given by $\Fr\cdot{}^\s\bbr$. We similarly define the $\bbr$-submodule  ${}^\s\bbr^m_\Fr$ of $\bbr^m_\Fr$. We will call these the $\bsm{\Fr}$\textbf{-standard vectors (in $\bsm{\bbr^m_\Fr}$)}. If $\Fr,\Fs\in\rz\bbr_+$, we will let $\Fs=\Fo(\Fr)$ denote the statement $\Fs/\Fr\sim 0$  and let $\bbr_{o(\Fr)}$ denotes those $\Fs$ with $\Fs=\Fo(\Fr)$ (we include $0$ here by convention). Given this, we clearly have the decomposition $\bbr_\Fr={}^\s\bbr_\Fr+\bbr_{o(\Fr)}$ with ${}^\s\bbr_\Fr\cap\bbr_{o(\Fr)}=\{0\}$.
    In particular, there is a surjective ring  homomorphism $\bsm{\fst_\Fr}:\bbr_\Fr\ra{}^\s\bbr_\Fr$, the $\bsm{\Fr}$\textbf{-standard part map} satisfying $\fst_\Fr$ is the identity on ${}^\s\bbr_\Fr$. Note that the kernel of the map is clearly $\bbr_{o(\Fr)}$. Clearly, also we have the $\rz\bbr_{nes}$-module version of the above, ie., a split exact sequence of $\rz\bbr_{nes}$-modules.
 \begin{align}
      \xymatrix { {}^\s\bbr^m_{\Fo(\Fr)} \ar@{^{(}->}[r]  &\bbr^m_\Fr \ar@{->>}[r]^{\fst_\Fr} &{}^\s\bbr^m_\Fr }
 \end{align}

     If $\Fu\in\rz\bbr^m$, let $\rz\bbr^m_\Fr(\Fu)=\{\Fv+\Fu:\Fv\in\rz\bbr^m_\Fr\}$. Note that $\rz\bbr^m_\Fr(\Fu)$ has the property that if $\a,\b\in\rz\bbr_{nes}$ with $\a+\b=1$ and $\Fv,\Fw\in\rz\bbr^m_\Fr(\Fu)$, then $\a\Fv+\b\Fw\in\rz\bbr^m_\Fr(\Fu)$. Hence $\rz\bbr^m_\Fr(\Fu)$ will be called an $\bsm{\Fr}$\textbf{-affine subspace} of $\rz\bbr^m$. In the usual way (via the transfer of the canonical standard affine identification $\Fu+\Fv\mapsto\Fv$) one can identify the $\Fr$-almost affine subspace $\bbr^m_\Fr(\Fu)$ with the $\Fr$-almost affine subspace $\bbr^m_\Fr$.
    Suppose that $\ccrs{f}:(\rz\bbr^m,0)\ra(\rz\bbr^n,0)$ is an internal function. We say that $\ccrs{f}$ is $\bsm{\Fr}$\textbf{-almost linear} if for all $\a,\b\in\rz\bbr_{nes}$ and $\Fv,\Fw\in\rz\bbr^m_\Fr$, we have that $\ccrs{f}(\a\Fv+\b\Fw)-\a\ccrs{f}(\Fv)-\b\ccrs{f}(\Fw)=o(\Fr)$.
 \end{definition}
 \begin{remark}\label{rem: *linear c= alm lin}
    Note that arbitrary *linear internal maps are $\Fr$-almost linear for all $\Fr$, but don't send $\bbr^m_\Fr$ into $\bbr^n_\Fr$. If such an $\ccrs{f}$ send a $\rz\bbr_{nes}$-basis of $\bbr^m_\Fr$ into $\bbr^n_\Fr$, then we do have $\ccrs{f}(\bbr^m_\Fr)\subset\bbr^n_\Fr$.
    In the case that $\Fr=1$, then $\ccrs{f}$ is $1$-almost linear implies that restricted to $\bbr^m_1=\rz\bbr^m_{nes}$  it's graph is infinitesimally close to a (possibly nonstandard) affine subspace, ie., it's standard part is an affine subspace (with possibly vertical subspaces).
 \end{remark}
   Since  standard functions (eg., our function $B$ below) will typically not satisfy $f(v_0)=v_0$  ($v_0$ being the point in the domain where we are testing for differentiability of $f$), we will need the corresponding nearness notion for affine maps. First, note that if we are looking at an internal map $\ccrs{f}:\rz\bbr^m\ra\rz\bbr^n$ restricted to  $\bbr^m_\Fr(\Fu_0)$, then the statement in the previous paragraph implies that this restriction can be considered as a map on $\bbr^m_\Fr$.
  If $\Fu_0\in\rz\bbr^m$, $\Fv_0\in\rz\bbr^n$ and $\ccrs{f}(\bbr^m_\Fr(\Fu_0))\subset\bbr^n_\Fr(\Fv_0)$, we say that $\ccrs{f}$ is \textbf{$\bsm{\Fr}$-almost affine at $\bsm{\Fu}$} if $\ccrs{f}(\a\Fv+\b\Fw)-\a\ccrs{f}(\Fv)-\b\ccrs(\Fw)=o(\Fr)$ holds for all $\a,\b\in\rz\bbr_{nes}$ with $\a+\b=1$ and $\Fv,\Fw\in\rz\bbr^m_\Fr$.
   Clearly, the sum of two $\Fr$-almost affine maps (defined on $\bbr^m_\Fr(\Fu)$ for some $\Fu$) is also $\Fr$-almost affine (with a different range). There are many other elementary properties of an $\Fr$-affine category (and relations between $\Fr$-affine and $\Fs$-affine categories) that can be straightforwardly fleshed out, but we will only develop those tools needed here.
 \begin{lemma}
   Suppose that $\SA:\rz\bbr^m\ra\rz\bbr^n$ is $\Fr$-almost affine (on $\bbr^m_\Fr(\Fu_0)$) and $\Ft_0=\SA(\Fu_0)$. Considering $\SA$ as a map on $\bbr^m_\Fr$ via the above identification, we have that $\SA-\Ft_0$ is $\Fr$-almost linear.
   In particular, suppose that $\SA$ is $\Fr$-almost affine. Considering $\SA$ as a map on $\bbr^m_\Fr$, we have that if $\SA(0)=0$, then $\SA$ is $\Fr$-almost linear. In particular, $\Fr$-almost affine maps are just internal translates of $\Fr$-almost linear maps.
 \end{lemma}
 \begin{proof}
   Our proof of the first statement is essentially the usual proof that an affine function fixing the origin is linear. Letting $\SL=\SA-\Ft_0$, we must first verify that for $\xi,\z\in\bbr^m_\Fr$, $\SL(\xi+\z)\stackrel{\Fr}{\sim}\SL(\xi)+\SL(\z)$. Using the definition of $\Fr$-almost affine in the case of a *affine sum with three terms ie., $\a+\b+\g=1$, in the case where $\a=\b=1$ and $\g=-1$, we get
    \begin{align}
      0=\SL(0)=\SL(\a\xi+\b\z-\g(\xi+\z))\stackrel{\Fr}{\sim}\SL(\xi)+\SL(\z)-\SL(\xi+\z).
    \end{align}
   We must second verify that, for $\la\in\rz\bbr_{nes}$ and $\Fv\in\bbr^m_\Fr$, $\SL(\la\Fv)\stackrel{\Fr}{\sim}\la\SL(\Fv)$. In this case, we again use three term affine sums $\a\xi+\b\z+\g\s$ where $\a+\b+\g=1$. That is, we apply $\Fr$-almost affineness in the case where  $\a=1,\b=-\la,\g=\la$ and $\xi=\la\Fv,\z=\Fv$ and $\s=0$ to get
     \begin{align}
       0=\SL(\la\Fv-\la\Fv+0)\stackrel{\d}{\sim}\SL(\la\Fv)-\la\SL(\Fv).
     \end{align}
   Clearly, the second statement in the lemma follows from the first.
 \end{proof}

   If $\ccrs{f}:\rz\bbr^m\ra\rz\bbr^n$ is internal and $\Fr\in\rz\bbr$ is  positive, we define the $\bsm{\Fr}$\textbf{-dilation} of $\ccrs{f}$ to be the map $\Fr^{-1}\circ\ccrs{f}\circ\Fr:\rz\bbr^m\ra\bbr^n$, ie., the map $\Fv\mapsto \Fr^{-1}\ccrs{f}(\Fr\Fv)$.
   An $\bsm{\Fr}$\textbf{-disk}  in $\bbr^m_\Fr(\Fu)$ is a *open, *convex subset $\SD\subset\bbr^m_\Fr(\Fu)$ of the form $\Fr\cdot\!\!\rz\! D+\Fu_0$ where $D\subset\bbr^m$ is convex, open and bounded.
   The following lemma is essentially tautalogical; nonetheless, it is  included due to its importance in our argument.
 \begin{lemma}\label{lem: dilation result}
    Suppose that $\ccrs{f}:\rz\bbr^m\ra\rz\bbr^n$ is $\Fr$-almost affine on an $\Fr$-disk $\Fr\cdot\!\!\rz\! D+\Fu\subset\bbr^m_\Fr(\Fu)$. Then ${}^\Fr\ccrs{f}=\Fr^{-1}\circ\ccrs{f}\circ\Fr$ is $1$-almost affine on $\rz D+\Fr^{-1}\cdot\Fu$.
 \end{lemma}
 \begin{proof}
    By the previous lemma, without loss of generality assume that $\ccrs{f}$ is $\Fr$-almost linear. We must show that for all $\a,\b\in\rz\bbr_{nes}$ and $\xi,\z\in\rz\bbr^m_{nes}$, we have
      \begin{align}
        {}^\Fr\ccrs{f}(\a\xi+\b\z)-\a\;{}^\Fr\ccrs{f}(\xi)-\b\;{}^\Fr\ccrs{f}(\z)=o(1).
      \end{align}
    Writing $\xi=\ov{\xi}/\Fr$ and $\z=\ov{\z}/\Fr$ for some $\ov{\xi},\ov{\z}\in\bbr^m_\Fr$, and noting that $\ov{\xi}\mapsto \ov{\xi}/\Fr$ is a bijection $\bbr^m_\Fr\ra\bbr^m_{nes}$, we see that the previous expression holds if and only if
      \begin{align}
        \Fr^{-1}\left[\ccrs{f}(\Fr(\a\;\ov{\xi}/\Fr+\b\ov{\z}/\Fr))-\a\ccrs{f}(\Fr(\ov{\xi}/\Fr))-\b\ccrs{f}(\Fr(\ov{\z}/\Fr))\right]=o(1).
      \end{align}
   for all $\ov{\xi},\ov{\z}\in\bbr^m_\Fr$.
   Noting that for a vector $\Fv\in\rz\bbr^n$, we have $\Fr^{-1}\Fv=o(1)$ if and only if $\Fv=o(\Fr)$, we see that the previous expression is equivalent to
     \begin{align}
        \ccrs{f}(\a\ov{\xi}+\b\ov{\z})-\a\ccrs{f}(\ov{\xi})-\b\ccrs{f}(\ov{\z})=o(\Fr),
     \end{align}
   for all $\a,\b\in\rz\bbr_{nes}$ and $\ov{\xi},\ov{\z}\in\bbr^m_\Fr$, as we wanted.
 \end{proof}
 \begin{remark}\label{rem: *aff -> alm aff}
   Note that if $\SA:\rz\bbr^m\ra\rz\bbr^n$ is *affine and $\ccrs{f}$ is $\Fr$-almost affine at $\Fu$, then $\ccrs{f}+\SA$ is $\Fr$-almost affine at $\Fu$.
 \end{remark}

\subsection{Nonstandard criterion for differentiability}

   We begin with a general fact connecting differentiability given the setup in the previous part. The following facts follow essentially from basics contained in Stroyan and Luxemburg, \cite{StrLux76} and an analog is stated and proved in another form in the author's work on the inverse function theorem, \cite{McGaffeyInverseMappingThm2012}.
   The following definition and proposition are stated in stronger forms than needed in this paper. The full strength will be needed in the following paper.

 \begin{definition}
   We say that $\ccrs{f}:\rz\bbr^m\ra\rz\bbr^n$ is \textbf{$\bsm{\Fr}$-almost affine at $\bsm{\Fu_0}$ stably for all positive infinitesimals} $\bsm{\Fr}$ if the following holds. There is a linear $L:\bbr^m\ra\bbr^n$ such that $\ccrs{f}$ satisfies the following for all positive infinitesimals $\Fr$. The  map restricted to $\bbr_\Fr(\Fu_0)$, ie.,  $\ccrs{f}:\bbr^m_\Fr(\Fu_0)\ra\rz\bbr^m$ is  $\Fr$-almost affine at $\Fu_0$ such that the $\Fr$-standard part of the $\Fr$-almost linear part of $\ccrs{f}$ exists and is $L$.
 \end{definition}
    If $f:\bbr^m\ra\bbr^n$, $\Fu\in\rz\bbr^m$ and $\Fr$ is a positive infinitesimal, let $\bsm{f^\Fr_{\Fu}}$ denote the internal map $\rz f$ restricted to $\bbr^m_\Fr(\Fu)$. If $\Fu=0$, we write $f^\Fr$ for $f^\Fr_0$.

 \begin{proposition}\label{prop: NS criterion for differen of f}
    Suppose that $f:\bbr^m\ra\bbr^n$ and $x_0\in\bbr^m$. Then the following are equivalent.
   \begin{enumerate}
     \item $f$ is differentiable at $\Fu_0$.
     \item $f^\Fr_{\Fu_0}$ is $\Fr$-almost affine at $\Fu_0$  stably for all  positive infinitesimals $\Fr$.
   \end{enumerate}
 \end{proposition}
 \begin{proof}
    Suppose that $f$ is differentiable at $x_0$ and let $L:\bbr^m\ra\bbr^n$ denote its derivative there. Let $\Fr_0$ be a positive infinitesimal. Then we clearly have  that if $\Fv\in\bbr^m_\Fr$, then $\rz f(x_0+\Fv)=f(x_0)+L(\Fv)+o(\Fr)$. In particular, if $\Fw\in\bbr^m_\Fr$ also and we have nearstandard $\a,\b$ with $\a+\b=1$, then
     \begin{align}\label{eqn: f is alm.affine eqn}
       \rz f(\a\Fv+\b\Fw)=f(x_0)+L(\a\Fv+\b\Fw)+o(\Fr).
     \end{align}
     Similarly, we have \textbf{(1)} $\a\;\rz f(\Fv)=\a f(x_0)+\a\; L(\Fv)+\Fo(\Fr)$ and \textbf{(2) }$\b\;\rz f(\Fv)=\b f(x_0)+\b L(\Fv)+\Fo(\Fr)$. Subtracting (1) and (2) from   \autoref{eqn: f is alm.affine eqn}, the linearity of $L$ implies
      \begin{align}
        \rz f(\a\Fv+\b\Fw)-\a\rz f(\Fv)-\b\;\rz f(\Fw)=f(x_0)-(\a +\b)f(x_0)+\Fo(\Fr),
      \end{align}
     and so $\a+\b=1$ finishes the first half of the proof.

     Now suppose that $\rz f$ is $\Fr$-almost affine at $\Fu_0$ stably for all $\Fr$ with (standard) linear map $L$. This just says
     for each positive infinitesimal $\Fr$ and $\Fv\in\bbr^m_\Fr$, we have $\rz f(x_0+\Fv)=f(x_0)+\rz L(\Fv)+\Fo(\Fr)$. That is, fixing $\Fr$, we have
     \textbf{(3)}: $\Fv\in\bbr)_\Fr\Rightarrow| \f{1}{\Fr}(\rz f(x_0+\Fv)-f(x_0))-\rz L(\Fv)|=\Fo(1)$.
     We need to make internal statements in order to construct a sufficiently consequential overflow. The following statements (special restrictions of the previous) will be sufficient. Let $U_\Fr$ denote the ball consisting of those  $\Fv\in\bbr_\Fr$ with $\|\Fv\|\leq\Fr$. Then,  for all $0<\Fr\sim 0$, (3) certainly implies the weaker assertion
      \begin{align}
         \Fv\in U_\Fr\;\Rightarrow\;\Big|\f{\rz f(x_0+\Fv)-\rz f(x_0)}{\Fr}-\rz L(\Fv)\Big|=\Fo(1).
      \end{align}
     The argument is finished as  follows. Replacing $=\Fo(1)$ by $< c$ for an arbitrary standard positive number $c$, we get an internal statement $S(\Fr,c)$ which holds for all positive infinitesimals $\Fr$ and hence for some positive standard $b$ by overflow. But we therefore have the statement: for every positive real $c$, there is positive real $b$ such that $S(b,c)$ holds, the criterion for differentiability at $x_0$.
 \end{proof}
\section{Nowhere differentiability of the blancmange function}
\subsection{Preliminaries}
   The blancmange function is defined as follows. (See Katz and Tall's paper for a conceptual discussion and Thim's paper for a conventional proof.) First define $s$ on the unit interval by $s(t)=t$ for $0\leq t\leq 1/2$ and $s(t)=1-t$ for $1/2<t< 1$ and extend $s$ to a function on all of $\bbr$ by defining it to have period $1$; ie., for all $j\in\bbz$ and $t\in [0,1)$ define $s(t+j)=s(t)$. By definition, $s$ is piecewise linear and continuous. Next, define it's dyadic dilations as follows. For $k\in\bbn$ and $t\in\bbr$, let $s_k(t)=s(2^k t)/2^k$. Finally, define, for $n\in\bbn$ and $t\in\bbr$
  \begin{align}
    B_n(t)=\sum_{j=0}^{n-1}s_k(t)\quad \text{and}\quad B(t)=\lim_{n\ra\infty}B_n(t).
  \end{align}
   It's clear that the above limit exists and is continuous as  $|s_k(t)|\leq 2^{k+1}$ for all $t\in\bbr$; and so is a  uniform limit of continuous functions on $[0,1]$. Letting $\bsm{B^n(t)=B(t)-B_n(t)}=\sum_{k=n}^\infty s_k(t)$, it's easy to verify the following critical facts.
 \begin{lemma}\label{lem: *fcnal eqn and *affine on interval}
    For each $n\in\bbn$ and $t\in\bbr$, we have the following functional equation $B(t)=B_n(t)+B(2^nt)/2^n$. For each $n\in\bbn$, the function $B_n$ is an affine function on the interval $(j2^{-n+1},(j+1)2^{-n+1})$ for all $j\in\bbz$.
 \end{lemma}
 \begin{remark}\label{rem: overlap}
   It's easy to see that if $C:\bbr\ra\bbr$ is another function and there is $x\in\bbr$ and $k\in\bbn$ with $C$ affine on $x+(k2^{-n+1}, (k+1)2^{-n+1})$, then for some interval $I\subset(k2^{-n+1},(k+1)2^{-n+1})$ of length at least $2^{-n}$, $B_n+C$ is affine on $x+I$.
   Although it is not relevant for our proof, the ``slope'' of $B_n$ on the interval $(j2^{-n+1},(j+1)2^{-n+1})$ ia a function (of $n$) and the dyadic expansion of the integer $j$; it will have values given by an integer between $-n$ and $n$.
 \end{remark}
\subsection{Proof of nowhere differentiability}
 \begin{theorem}\label{thm: 1st blancmange result}
    The blancmange function is a continuous function that is differentiable at no point in $\bbr$.
 \end{theorem}
 \begin{proof}
    Suppose, by way of contradiction, that $B$ is differentiable at $t_0\in\bbr$. Let $\om\in\rz\bbn_\infty$, so that $\d=1/2^\om$ is a positive infinitesimal. Now $B$ is differentiable at $t_0$ implies that $\rz B$ is $\d$-almost affine on $\rz\bbr_\d(t_0)=t_0+\sqcup_{k\in\un{\bbn}}[k\d,(k+1)\d)$. Also by the transfer of \autoref{lem: *fcnal eqn and *affine on interval}, we have that $\rz B_\om$ is *affine on $(\Fl\d/2,(\Fl+1)\d/2)$ for all $\Fl\in\rz\bbz$ (and so eg., $\d$-almost affine in each of these intervals, see  \autoref{rem: *aff -> alm aff}). And \autoref{rem: overlap} then says there is an interval $I\subset(0,\d)$ of length at least $\d/2$ so that $\rz B- B_\om$ is $\d$-almost affine on $t_0+I$. Hence, this and the transfer of the functional equation  gives that $\d\circ\rz B\circ\d^{-1}=\rz B- B_\om$ is $\d$-almost affine on $t_0+I$. But the dilation lemma,  \autoref{lem: dilation result}, applied in dimension $1$, then implies that $\rz B$ is $1$-almost affine on $t_0/\d+\d^{-1}I$, which is an interval of length at least $1/2$, an absurdity by  \autoref{rem: *linear c= alm lin} and as $B$ is a continuous function that is not affine on any interval of length $1/2$.
 \end{proof}
 \begin{remark}
    First, note that this argument cannot work if the dilation of domain and range is not a conjugation automorphism; eg., if it is not the identity operator on the linear part of affine maps. In particular, our argument fails if we consider $s_k(t)=s(a^kt)/b^k$ for $b>a$. In fact, such functions are often differentiable, eg., see the paper of Thompson and Hagler, \cite{ParabolicTakagi}, where the authors show that, in the case $a=2,b=4$, $B$ is just part of a parabolic curve!
    Next, our construction shows that $B$ fails the nonstandard test for differentiability in a very big way. That is, it fails the test for $\d$-almost linearity for $\d=1/2^\om$ for \textbf{all} $\om\in\rz\bbn_\infty$. This is not too surprising as $B$ is standard.
 \end{remark}
\section{Generalized blancmange functions}
\subsection{Construction of $\SB$}\label{subsec: construct SB}
    Here we will see that our proof, with minor alterations, works for very large families of analogously defined functions. First, instead of the continuous piecewise affine function $s$, we will now have an open subset $\SS$ of an infinite dimensional vector space of such piecewise affine continuous function, where the function $s$ of the previous section is essentially the simplest element of this set $\SS$. (As this vector space will not play a roll here, we will leave its description to a later paper.) Further, for a given $s\in\SS$, instead of the single sequence of functional equations (generating $B$) $s_k(t)=s(2^k)/2^k$ for $k=0,1,2,\cdots$, we will have a one parameter family of such sequences $s_k(t)=s(b^kt)/b^k$ for $2<b\in\bbn$ a multiple of an integer determined by $s$. Hence, we will generate a quite large family of nowhere differentiable functions, an issue we will address after our theorem.
    In the following, hopefully the reader should see how the previous proof is very close to our proof below for these generalized blancmange functions.

   First of all, let's define an infinite  general family of generating functions, $\bsm{\SS}$ for which our generator $s$ is a single instance. As before our generator $s$ will be defined on the interval $[0,1]$ so that it can be extended to a continuous function on $\bbr$ with a period $1$. Define $s(0)=s(1)=0$ and for some  $p\in\bbn$, if  $0< i<p$, let $s(i)\in\bbr$ be arbitrary with  $s(i_0/p)$ nonzero for some $i_0$. Given that $s$ is now defined at the points $i/p$ for $0\leq i\leq p$, extend $s$ to a function on all of $[0,1]$ by linear interpolation so that $s$ will be a continuous function on $[0,1]$ that is affine on each of the intervals $(i/p,(i+1)/p)$ for $0\leq i<p$. As $s(0)=s(1)=0$, we can extend $s$ to a continuous function on all of $\bbr$ by defining $s(j+t)=s(t)$ for $j\in\bbz\ssm\{0\}$ and $t\in[0,1)$. Let $\bsm{\SS_p}$ consist of the set of all such $s$ for our given integer $p>1$ and let $\bsm{\SS}$ denote the union of all $\SS_p$ as $p>1$ varies in $\bbn$. It is no problem that this is not a disjoint union.
   Note that the  $s$ defining our blancmange function has $p=2$ and $b=2$. For $s\in\SS_p$ and $c\in\bbn$, let $b=cp$ for some $c\in\bbn$ and, for $k\in\bbn\cup\{0\}$, define
    \begin{align}\label{eqn: def of little s}
      s_k(t)=s(b^kt)/b^k.
    \end{align}
   As before, for $n\in\bbn$, defining  $B_n(t)=\sum_{j=0}^{n-1}s_j(t)$, we find  that the sequence of piecewise continuous functions $B_n$ viewed on $[0,1]$ converge uniformly to a continuous function $B=\bsm{B(s,c)}$ on $[0,1]$. So applying  periodicity, we get uniform convergence on all of $\bbr$. Given this, for $s\in\SS_p$, let
    \begin{align}\label{eqn: def of SB}
      \bsm{\SB(s)}=\{B(s,c):c\in\bbn\},\;\; \bsm{\SB(p)}=\cup\{\SB(s):s\in\SS_p\}\\
      \text{and}\;\;\bsm{\SB}=\cup\{\SB(p):p>1\;\text{is an integer}\}\qquad\quad\notag
    \end{align}
   denote the set of all of these continuous functions defined by a given generating function $s\in\SG$ and compatible dilation factor $c\in\bbn$.
\subsection{Nowhere differentiability of elements of $\SB$}
   Given the above constructions, we need a pair of lemmas before we can prove nowhere differentiability of elements of $\SB$.
   We begin with a simple analog of \autoref{lem: *fcnal eqn and *affine on interval}.
 \begin{lemma}
   Assume that $p>2$ for our piecewise linear function $s$ defined above. For each $n\in\bbn$ and $t\in\bbr$, we have the functional equation $B(t)=B_n(t)+B(b^nt)/a^n$. For each $n\in\bbn$, the function $B_n$ is an affine function on the interval $(j/(pb^n),(j+1)/(pb^n))$ for all $j\in\bbz$.
 \end{lemma}
 \begin{proof}
   The functional equations are easy to verify as before. On the other hand, note that the vertices of the affine function $s_k(t)$ are the points $V_k=\{i/(pb^k):i\in\bbz\}$. In particular, as $V_j\subset V_k$ for $j\leq k$, then all of the functions $s_k$ for $k\leq n-1$ are affine on each of the intervals that $s_{n-1}$ is (ie., those of the form $(i/(pb^{n-1}),(i+1)/(pb^{n-1}))$ for $i\in\bbz$). Therefore, the sum $s+s_1+s_2+\cdots+s_{n-1}$ is affine on each such interval.
 \end{proof}
  For a replay of our earlier proof to work, we need to prove, in contrast, that the function $B(s,b)$ is not affine on any interval of positive length in $[0,1]$, (a fact that is obvious with the original blancmange function). We prove this in the next lemma by a simple combinatorial argument. First, we need some notation. For a fixed $p>1$ in $\bbn$, $s\in\SS_p$ and $B=B(s,b)$, if  $n\in\bbn$, write $B=B_n+B^n$ where $B_n=\sum_{j<n}s_j$ and $B^n=\sum_{j\geq n}s_j$.
  \begin{lemma}\label{lem: B(s,b) not affine on I}
    Suppose that $s\in\SS$ with subdivision number $p\in\bbn$ and $b=cp$ for some $c\in\bbn$ and $B=B(s,b)$ is the function defined above. Then there is no open interval $I\subset[0,1]$ such that $B|I$ is affine.
  \end{lemma}
  \begin{proof}
    Suppose, to the contrary, such an interval $I\subset[0,1]$ exists. That is, all points on the graph of $B$ on the interval $I$ are colinear. Then, there is a minimum $m\in\bbn$ and some $j$ so that
     \begin{align}
       I_{m,j}\dot=(\f{j}{pb^m},\f{j+1}{pb^m})\subset I.
     \end{align}
    Now one can check the following facts. (1) We have $s_{m+1}(\f{j}{pb^m})=s_{m+1}(\f{j+1}{pb^m})=0$, but $s_{m+1}(\f{j_0}{pb^{m+1}})\not=0$ for some $\f{j_0}{pb^{m+1}}\in I_{m,j}$.
    (2) For all positive integers $k\leq m$, $s_k$ is affine on $I_{m,j}$. (3) For all integers $k\geq m+2$, $s_k(\f{j}{pb^m})=s_k(\f{j+1}{pb^m})=s_k(\f{j_0}{pb^{m+1}})=0$.
    Given these three facts, we can deduce the following. First, (1) and (2) clearly imply (4): the points $B_{m+2}(\f{j}{pb^m})$, $B_{m+2}(\f{j+1}{pb^m})$ and $B_{m+2}(\f{j_0}{pb^{m+1}})$ are not colinear on the graph of $B_{m+2}$.
    On the other hand, fact (3) implies that (5): $B^{m+2}(\f{j}{pb^m})=B^{m+2}(\f{j+1}{pb^m})=B^{m+2}(\f{j_0}{pb^{m+1}})=0$. Clearly then, as $B(t)=B_{m+2}(t)+B^{m+2}(t)$ for all $t$,  facts (4) and (5) imply that the points $B(\f{j}{pb^m})$, $B(\f{j+1}{pb^m})$ and $B(\f{j_0}{pb^{m+1}})$ on the graph of $B$ are not colinear. As these points lie in the part of the graph over $I$, we have a contradiction.
  \end{proof}
   We can now verify our assertion.
 \begin{theorem}\label{thm: 2nd blancmange result}
   Suppose that $s\in\SG_p$ is one of our generating functions with $p>2$, and $B\in\SB(s,c)$ for a given $c\in\bbn$. Then $B$ is continuous and nowhere differentiable.
 \end{theorem}
 \begin{proof}
    We just need to prove nowhere differentiability. Suppose, to the contrary, that $B$ is differentiable at some $t_0\in (0,1)$. Let $\om\in\rz\bbn$ be an infinite integer and let $\d=1/(pb^{\om-1})$. As $\d$ is a positive infinitesimal, then our contrary hypothesis implies that $\rz B$ is $\d$-almost affine on $\bbr_\d(t_0)\supset\sqcup\{t_0+(k\d,(k+1)\d):k\in\bbz\}$; eg., on $t_0+(0,\d)$. We also have that the transfer of \autoref{lem: *fcnal eqn and *affine on interval} (or statement (2) in the previous lemma) evaluated at $\om$ implies that $\rz B_\om$ is *affine on $(\Fj\d,(\Fj+1)\d)$ for all $\Fj\in\rz\bbz$.
    But there is $\Fj_0\in\rz\bbz$ so that $(\Fj_0\d,(\Fj_0+1)\d)$ and $t_0+(0,\d)$ intersect in an interval $\SI$ of length at least $\d/2$. That is,\linebreak  $\d\circ\rz B\circ\d^{-1}=\rz B-B_\om$ is $\d$-almost affine on $\SI$. Hence, the dilation lemma, \autoref{lem: dilation result}, implies that $\rz B$ is $1$-almost affine on $\d^{-1}\SI$, a *interval in $\rz\bbr$ of length at least $1/2$. That is, as $\rz B$ has *period $1$, then $B=\fst(\rz B)$ must be affine on an interval of length at least $1/2$.  Our contradiction then follows from  \autoref{lem: B(s,b) not affine on I}.
 \end{proof}
 \begin{remark}
    Note that our proof seems capable of giving the same conclusion with a weaker hypothesis. That is, although differentiability of $B$ at $t_0$ implies that $\rz B$ is $\d$ almost affine on all of $\bbr_\d(t_o)$, we arrived at our conclusion using the $\d$ almost affineness of $B$ only on the small segment $t_0+(0,\d)$ of $\bbr_\d(t_0)$. Further, the assumption of $\d$ almost affineness on any of the segments $t_0+(k\d,(k+1)\d)$ would yield a contradiction by the same argument. This seems to imply, for example, that $B$ does not even have one sided derivatives at any $t_0\in[0,1]$. These implication will be pursued in a later paper.
 \end{remark}
\subsection{Density of $\SB$}
    There are a fair number of constructions of continuous nowhere differentiable functions in terms of continuous piecewise affine functions. Detailed description of this work occurs in Thim's work, \cite{ContNowhereDifferenFcns}. A more limited, but more graphic display of such functions can be found in Google images under the keywords ``nowhere differentiable'', ``Weierstrass function'', ``Takagi function'', et cetera.  Beyond the blancmange function, our family of functions $\SB$ includes some described in Thim's paper, but also includes many not yet described. For example, $\SB$, includes functions generated by elements $s\in\SS$ with arbitrarily small support. Furthermore, we can also choose $s\in\SS$ that are arbitrarily Lipschitz close to eg., $\sin(\pi x)$. (For crude, but hopefully suggestive, examples of both, see the last section.)  In fact, our family is sufficiently numerous to uniformly approximate any continuous functions $f:[0,1]\ra\bbr$ sending $0$ and $1$ to $0$. Let $C^0$ denote this set of continuous functions on $[0,1]$. First, we have a lemma that is a slight generalization of Lemma 4.3 in Thim. One might notice, that besides being distinctly shorter than his proof (see pages 74-75 of his text), we use no estimates, only simple order of magnitude arguments made available by nonstandard methods. Although he claims the proof is taken essentially from Oxtoby's classic text, \cite{Oxtoby}, we could not find the relevant text in Oxtoby.
    For a function $g:[0,1]\ra\bbr$, let $\|g\|$ denote $\sup\{|g(t)|:t\in[0,1]\}$, the supremum norm. We will also use this notation in the internal realm.
 \begin{lemma}
    $\SS$ is dense in $C^0$ with respect to the uniform norm.
 \end{lemma}
 \begin{proof}
    Let $f\in C^0$ and $E\subset\bbr_+$ denote the set of $e\in\bbr_+$ such that there is $s\in\SS$ with $\|f-s\|<e$. It suffices to prove that $\rz E$ contains infinitesimals. So we just need to show that there is $\ccrs{s}\in\rz\SS$ with $\|\rz f-\ccrs{s}\|\sim 0$. Choose $\pzp\in\rz\bbn_\infty$ with $[\pzp]=\{0,1,\cdots,\pzp-1,\pzp\}$. Let $\SP=\{\Fj/\pzp:\Fj\in[\pzp]\}$ and $\SI_\Fj=[\Fj/\pzp,(\Fj+1)/\pzp]$, a *compact interval. Define $\ccrs{s}(\Fj/\pzp)=\rz f(\Fj/\pzp)$ for all $\Fj\in[\pzp]$ extending it to be *affine on each $\SI_\pzp$. Clearly, $\ccrs{s}\in\rz\SS$. Fixing $\Fj\in[\pzp]\ssm\{\pzp\}$, by standard continuity of $f$, $\rz f(\Ft)\sim\rz f(\Fj/\pzp)$ for all $\Ft\in\SI_\Fj$ and so by *affineness of $\ccrs{s}$ on $\SI_\Fj$, we have $\ccrs{s}(\Ft)\sim\ccrs{s}(\Fj/\pzp)$. Put together, these say that $\rz f(\Ft)\sim\ccrs{s}(\Ft)$. That is, $\rz f(\Ft)-\ccrs{s}(\Ft)\sim 0$ for all $\Ft\in\SI_\Fj$. As $\SI_\Fj$ is *compact and $\rz f-\ccrs{s}$ is *continuous, $\e_\Fj=\rz \max\{|\rz f(\Ft)-\ccrs{s}(\Ft)|:\Ft\in\SI_\Fj\}$ exists and is infinitesimal. But $\SA=\{\e_\Fj:\Fj\in\{0,1,\cdots,\pzp-1\}\}$ is a *finite (eg., internal) set of infinitesimals and so $\rz\max\SA\in\SA$, eg., is infinitesimal.
 \end{proof}
   From the above lemma, we have our assertion.
 \begin{corollary}\label{cor: SB is dense}
    $\SB$ is dense in $C^0$ in the uniform topology.
 \end{corollary}
 \begin{proof}
    By the above lemma, it suffices to verify that for a fixed $s\in\SS_p$, there is $\scrb\in\rz\SB(s)$ with $\|\rz s-\scrb\|\sim 0$. Choose $\Fc\in\rz\bbn_\infty$, with $\scrb=\rz B(s,\Fc)$. Now $\Fb=\Fc\cdot p\in\rz\bbn_\infty$, and so letting $M=\|s\|$, we have for all $\Ft\in\rz [0,1]$ that
     \begin{align}
       |\rz s(\Ft)-\rz B(s,\Fc)(\Ft)|\leq \f{M}{\Fb}\cdot\rz\sum_{\Fj=0}^{\infty}\Fb^{-\Fj}\sim 0
     \end{align}
    as we wanted.
 \end{proof}

\subsection{Perspective}
     In order to prove the above results, we only needed the following facts. First, we needed a (fairly crude) nonstandard characterization of differentiability at a point $t_0$, ie., that for all positive infinitesimals $\Fr$, the function restricted to $\bbr_\Fr(t_0)$ is $\Fr$-almost affine. Second, we needed the fact that dilation sends almost affine maps to almost affine maps. Finally, third we needed the transfer of the set of functional equations as well as the fact that approximations were affine on sufficiently large intervals. In particular, we did not need nuanced versions of the nonstandard characterizations of differentiability.
     Although such a transcription is theoretically possible, from the author's perspective, a rewriting of this proof in standard language would seem to be a nontrivial task. One must standardize our strategy: we fixed an infinite index and did some fairly detailed combinatorics on the geometric configurations existing at that index.

   In our second installment, we will consider functions not generated in terms of functional equations and will use an alternative nonstandard characterization of differentiability at a point. More specifically, for a function to be differentiable at a point $t_0$, not only does $\rz f$ need to be $\d$ almost affine on $\bbr_\d$ for all infinitesimal scales, $\d$; but crudely, dilation from one infinitesimal scale to another carries our almost affine maps into each other.

\section{A computational view of elements of $\SB$}

   Due to the constructive nature and broad types of behavior of these functions, the author decided to investigate some computer visualization schemes with hopes of getting some  insight into the natures of these (continuous) nowhere differentiable functions. Others, eg., Thompson and Hagler, \cite{ParabolicTakagi}, have used numerical computational tools in attempting to gain insight into continuous nowhere differentiable functions; in fact, going at least as far back as the 1961 work of Salzer and Levine, \cite{TableWeierstrassFcn1961}.  After weeks of investigations (of Tikz, Gnuplot, Sage and other open source tools), the author decided the open source Python suite (\href{https://code.google.com/p/pythonxy/wiki/Welcome}{python(x,y)}) of abstract computational and graphing tools was best suited for this goal.  The author invested two months to learn sufficient python (and matplotlib) syntax to construct a piece of code allowing at least a multiscaled impressionistic view of these functions.

   We have two versions of the code. After compiling, both yield a full page with six  coordinate chart ``snapshots''. Each of the first five snapshots is followed by another that is a magnification (around a fixed magnification point) of the graphs on the previous coordinate chart. Each coordinate chart displays the same multicolored tuple of graphs of approximations $B_{N_1},\cdots,B_{N_k}$ of a given element $B\in\SB$. Among other parameter choices, the user can choose the center of magnification, the magnification factor and the choice of the tuple of $N_1,\cdots,N_k$, although the author has constructed the coordinate legend for a tuple of length  six or less. (The legend is not totally debugged. It's off screen on some displays, but can be pulled in using the hspace toggle of the subplot configuration tool.)  Of the two code choices, the first can be copied to an interactive console (we used Spyder lite) where it runs with little prompting. To run the code with different parameters, one must manually alter the code at eg., the number and values of the vertices, the magnification point etc.  Alternatively, the second version is written to prompt for these parameters, eg., for  the vertices of the generator $s$, where $B=B(s,1)$ (see  \autoref{subsec: construct SB}), the  focal point, etc. After first saving the code as a python file ( by eg., copying it to the Spyder text editor which can save it properly on prompting), one can then ``run'' it on Spyder with the accompanying console prompting for the desired parameters.

     The above is an outline of \emph{our} procedure; in either case, hopefully the code is sufficiently clear (to one with an elementary knowledge of python) so that the prompts can be extended by  alterations of the code allowing a more refined sequence of magnifications of the tuple of approximations of the  given element of $\SB$ (for a reader who has at hand more computational power than the author's pedestrian resources). Furthermore, the author is struggling to build computationally more efficient code (eg., using python's multiprocessing module) and welcomes the input of pythonistas. Whatever the case, a reader who might be interested in viewing the sequence of graphs of a particular element of $\SB$ is welcome to  copies of the code from the author upon request at \href{mailto:mcgaffeythomas@gmail.com}{the author's gmail account}.

     We have structured the above discussed code precisely to probe the manner in which a sequence of approximations ``fall away'' one at a time as we continue to magnify, leaving the more intense core. On the other hand, as noted above, the family $\SB$ includes numerous examples whose graphs display novel properties.  Using a greatly simplified and redirected version of the code, we've included graphs of a pair of such examples in figure \autoref{image: two graphs} formatted for this article. We display the generator $s$ and the generalized blancmange function $B(s,c)$ with $c=1$ arising from it. The graphs are given in terms of the approximation $B_{12}=\sum_{j=0}^{11}s_j$ of our function $B$.  Recall that if $p\in\bbn$ and $s\in\SS_p$ is a generator for  $B=B(s,1)\in\SB$, then $s$ is defined by the $p+1$ values $v_j=s(j/p)$ for $0\leq j\leq p$. \emph{We will denote this by} $\bsm{s=[v_0,\cdots,v_p]}$ in the graphs below.  The first has a curious smooth look and the second a sparse quality. Obviously, $p$ here is a small integer; by making $p$ arbitrarily large (or $c$ a large integer) we can accentuate these behaviors greatly.

\newpage
\vfill

\begin{figure}[!h]
  \centering
    \includegraphics[width=5in,height=8.5in]{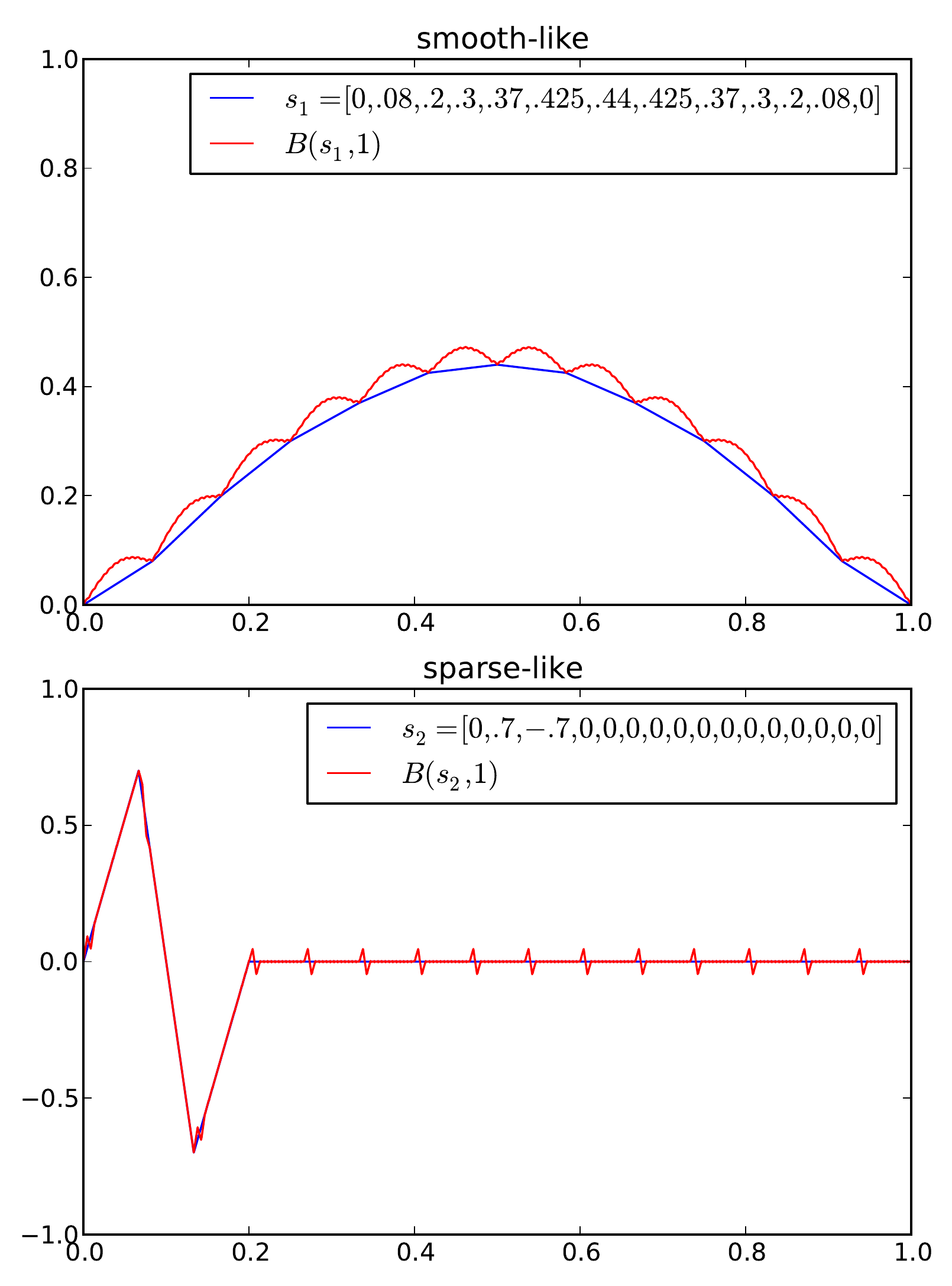}
   \caption{very different elements of $\SB$ \label{image: two graphs}}
\end{figure}

\vfill
\clearpage

\bibliographystyle{amsplain}
\bibliography{nsabooks}

\providecommand{\bysame}{\leavevmode\hbox to3em{\hrulefill}\thinspace}
\providecommand{\MR}{\relax\ifhmode\unskip\space\fi MR }
\providecommand{\MRhref}[2]{%
  \href{http://www.ams.org/mathscinet-getitem?mr=#1}{#2}
}
\providecommand{\href}[2]{#2}
\begin{thebibliography}{10}

\bibitem{TakagiSurvey2011}
Pieter~C. Allaart and Kiko Kawamura, \emph{The {T}akagi function: a survey},
  Real Anal. Exchange \textbf{37} (2011/12), no.~1, 1--54. \MR{3016850}

\bibitem{TenMisconceptionsPublished}
Piotr Blaszczyk, MikhailG. Katz, and David Sherry, \emph{Ten misconceptions
  from the history of analysis and their debunking}, Foundations of Science
  \textbf{18} (2013), no.~1, 43--74 (English).

\bibitem{QuestionsGeneralityChorley2009}
R.~Chorley, \emph{Questions of generality as probes into nineteenth century
  analysis}, {Handbook on Generality in Mathematics and the Sciences}
  (R.~Chorlay K.~Chemla and D.~Rabouin, eds.), forthcoming, 2009.

\bibitem{EulerAnalysisInfinite}
Leonhard Euler, \emph{Introducci\'on al an\'alisis de los infinitos}, Sociedad
  Andaluza de Educaci\'on Matem\'atica ``Thales'', Seville, 2000, Translated
  from the Latin by Jos{\'e} Luis Arantegui Tamayo, Annotated by Antonio
  Jos{\'e} Dur{\'a}n Guarde{\~n}o, With introductory material by Javier
  Ord{\'o}{\~n}ez, Mariano Mart{\'{\i}}nez P{\'e}rez and Dur{\'a}n
  Guarde{\~n}o, Edited by Dur{\'a}n Guarde{\~n}o and Francisco Javier P{\'e}rez
  Fern{\'a}ndez. \MR{1841792 (2002d:01014b)}

\bibitem{InfinitesAnalyGordonKusraevKutateladzeBk2002}
E.I. Gordon, A.G. Kusraev, and S.S. Kutateladze, \emph{{Infinitesimal
  Analysis}}, Mathematics and its applications, Kluwer Academic Publishers,
  2002.

\bibitem{ParabolicTakagi}
Alexander L.~Thompson III and James~N. Hagler, \emph{{Between a parabola and a
  nowhere differentiable place}}, unpublished, 20 pages, 2006.

\bibitem{Nonstandardanalysisaxiomatically}
Vladimir Kanovei and Michael Reeken, \emph{Nonstandard analysis,
  axiomatically}, Springer Monographs in Mathematics, Springer-Verlag, Berlin,
  2004. \MR{2093998 (2006e:03001)}

\bibitem{KatzTensionIntuitiveInfinitesFormalMath}
M.~G. {Katz} and D.~{Tall}, \emph{{Tension between Intuitive Infinitesimals and
  Formal Mathematical Analysis}}, ArXiv e-prints (2011).

\bibitem{Lindstrom1988}
Tom Lindstr{\o}m, \emph{{An invitation to nonstandard analysis}}, {Nonstandard
  Analysis and its Applications} (N.~Cutland, ed.), Cambridge University Press,
  1988, pp.~1--105.

\bibitem{McGaffeyInverseMappingThm2012}
T.~{McGaffey}, \emph{{Magnification Spaces: A nonstandard approach to inverse
  mapping theorems}}, ArXiv e-prints (2012).

\bibitem{Oxtoby}
John~C. Oxtoby, \emph{Measure and category}, second ed., Graduate Texts in
  Mathematics, vol.~2, Springer-Verlag, New York, 1980, A survey of the
  analogies between topological and measure spaces. \MR{584443 (81j:28003)}

\bibitem{TableWeierstrassFcn1961}
Herbert~E. Salzer and Norman Levine, \emph{Table of a {W}eierstrass continuous
  non-differentiable function}, Math. Comp. \textbf{15} (1961), 120--130.
  \MR{0122011 (22 \#12738)}

\bibitem{StrLux76}
K.D. Stroyan and W.A.J. Luxemburg, \emph{{Introduction to the Theory of
  Infinitesimals}}, Academic Press, 1976.

\bibitem{ContNowhereDifferenFcns}
J.~Thim, \emph{{Continuous Nowhere Differentiable Functions}}, Master's thesis,
  Lulea University of Technology, 2003.

\bibitem{NowhereDifferenBoasMR0274670}
P.~van Emde~Boas, \emph{Nowhere differentiable continuous functions}, Math.
  Centrum Amsterdam Afd. Zuivere Wisk. \textbf{1969} (1969), no.~ZW-012, 24.
  \MR{0274670 (43 \#432)}

\end{thebibliography}

\end{document}